\documentclass[11pt,a4paper]{article}
\usepackage{amsmath}
\usepackage{epsf,epsfig,amsfonts,amsgen,amsmath,amstext,amsbsy,amsopn,amsthm}
\usepackage{ebezier,eepic}
\usepackage{color}
\newtheorem{theorem}{Theorem}

\newtheorem{lemma}{Lemma}
\newtheorem{false statement}{False statement}

\theoremstyle{definition}

\newtheorem{claim}{Claim}

\newtheorem{case}{Case}

\begin{document}

\title{On star-wheel Ramsey numbers}

\author{{\sl Binlong Li}\footnote{The work is supported by
NSFC (No. 11271300), the Doctorate Foundation of Northwestern
Polytechnical University (No. cx201202) and the project NEXLIZ -
CZ.1.07/2.3.00/30.0038, which is co-financed by the
European Social Fund and the state budget of the Czech Republic.}\\
\small Department of Applied Mathematics\\
\small Northwestern Polytechnical University\\
\small Xi'an, Shaanxi 710072, P.R.~China\\
\small European Centre of Excellence NTIS \\
\small University of West Bohemia\\
\small 306 14 Pilsen, Czech Republic;\\
\small libinlong@mail.nwpu.edu.cn \and {\sl Ingo
Schiermeyer\footnote{Research was partly supported by
the DAAD-PPP project "Rainbow connection and cycles in graphs" with project-ID 56268242.}}\\
\small Institut f\"ur Diskrete Mathematik und Algebra\\
\small Technische Universit\"at Bergakademie Freiberg\\
\small 09596 Freiberg, Germany\\
\small Ingo.Schiermeyer@tu-freiberg.de}

\date{November 26, 2014}

\maketitle

\begin{abstract}
For two given graphs $G_1$ and $G_2$, the Ramsey number $R(G_1,G_2)$
is the least integer $r$ such that for every graph $G$ on $r$
vertices, either $G$ contains a $G_1$ or $\overline{G}$ contains a
$G_2$. In this note, we determined the Ramsey number
$R(K_{1,n},W_m)$ for even $m$ with $n+2\leq m\leq 2n-2$, where $W_m$
is the wheel on $m+1$ vertices, i.e., the graph obtained from a
cycle $C_m$ by adding a vertex $v$ adjacent to all vertices of the
$C_m$.

\medskip
\noindent {\bf Keywords:} Ramsey number; star; wheel
\smallskip

\noindent {\bf AMS Subject Classification:} 05C55, 05D10

\end{abstract}

\section{Introduction}

Throughout this paper, all graphs are finite and simple. For a pair
of graphs $G_1$ and $G_2$, the \emph{Ramsey number} $R(G_1,G_2)$, is
defined as the smallest integer $r$ such that for every graph $G$ on
$r$ vertices, either $G$ contains a $G_1$ or $\overline{G}$ contains
a $G_2$, where $\overline{G}$ is the complement of $G$. Note that
$R(G_1,G_2)=R(G_2,G_1)$. We denote by $P_n$ ($n\geq 1$) and $C_n$
($n\geq 3$) the path and cycle on $n$ vertices, respectively. The
bipartite graph $K_{1,n}$ ($n\geq 2$) is called a \emph{star}. The
\emph{wheel} $W_n$ ($n\geq 3$) is the graph obtained by joining a
vertex and a cycle $C_n$.

In this note we consider the Ramsey numbers for stars versus wheels.
There are many results on this area. Hasmawati \cite{Hasmawati}
determined the Ramsey number $R(K_{1,n},W_m)$ for $m\geq 2n$.

\begin{theorem}[Hasmawati \cite{Hasmawati}]\label{ThHa}
If $n\geq 2$ and $m\geq 2n$, then
$$R(K_{1,n},W_m)=\left\{\begin{array}{ll}
n+m-1,  & \mbox{if both } n \mbox{ and } m \mbox{ are even};\\
n+m,    & \mbox{otherwise}.
\end{array}\right.$$
\end{theorem}

So from now on we consider the case that $m\leq 2n-1$. For odd $m$,
Chen et al. \cite{ChenZhangZhang} showed that if $m\leq n+2$, then
$R(K_{1,n},W_m)=3n+1$. Hasmawati et al.
\cite{HasmawatiBaskoroAssiyatun} proved that the values remain the
same even if $m\leq 2n-1$.

\begin{theorem}[Hasmawati et al.
\cite{HasmawatiBaskoroAssiyatun}]\label{ThHaBaAs}
  If $3\leq m\leq 2n-1$ and $m$ is odd, then
  $$R(K_{1,n},W_m)=3n+1.$$
\end{theorem}

So it is remains the case when $m\leq 2n-2$ and $m$ is even.
Surahmat and Baskoro \cite{SurahmatBaskoro} determined the Ramsey
numbers of stars versus $W_4$.

\begin{theorem}[Surahmat and Baskoro \cite{SurahmatBaskoro}]\label{ThSuBa}
If $n\geq 2$, then
$$R(K_{1,n},W_4)=\left\{\begin{array}{ll}
2n+1,   & \mbox{if } n \mbox{ is even};\\
2n+3,   & \mbox{if } n \mbox{ is odd}.
\end{array}\right.$$
\end{theorem}

Chen et al. \cite{ChenZhangZhang} established $R(K_{1,n},W_6)$, and
Zhang et al. \cite{ZhangChenZhang,ZhangChengChen} established
$R(K_{1,n},W_8)$.

In this note we first give a lower bound on $R(K_{1,n},W_m)$ for
even $m\leq 2n-2$. One can check that when $m=6,8$, the lower bound
on $R(K_{1,n},W_m)$ in Theorem \ref{ThLower} is the exact value, see
\cite{ChenZhangZhang, ZhangChengChen, ZhangChenZhang}.

\begin{theorem}\label{ThLower}
If $6\leq m\leq 2n-2$ and $m$ is even, then
$$R(K_{1,n},W_m)\geq\left\{\begin{array}{ll}
2n+m/2-1,  & \mbox{if both } n \mbox{ and } m/2 \mbox{ are even};\\
2n+m/2,    & \mbox{otherwise.}
\end{array}\right.$$
\end{theorem}

Moreover, we establish the exact values when $n+2\leq m\leq 2n-2$.
We will show that the lower bound in Theorem \ref{ThLower} is the
exact value if $m\geq n+2$.

\begin{theorem}\label{ThExactly}
If $n+2\leq m\leq 2n-2$ and $m$ is even, then
$$R(K_{1,n},W_m)=\left\{\begin{array}{ll}
2n+m/2-1,  & \mbox{if both } n \mbox{ and } m/2 \mbox{ are even};\\
2n+m/2,    & \mbox{otherwise.}
\end{array}\right.$$
\end{theorem}

\section{Preliminaries}

We denote by $\nu(G)$ the order of $G$, by $\delta(G)$ the minimum
degree of $G$, $c(G)$ the circumference of $G$, and $g(G)$ the girth of
$G$, respectively. The graph $G$ is said to be \emph{pancyclic} if $G$ contains
cycles of every length between 3 and $\nu(G)$, and \emph{weakly
pancyclic} if $G$ contains cycles of every length between $g(G)$ and
$c(G)$.

We will use the following results.

\begin{theorem}[Dirac \cite{Dirac}]\label{ThDi}
Every 2-connected graph $G$ has circumference
$c(G)\geq\min\{2\delta(G),\nu(G)\}$.
\end{theorem}

\begin{theorem}[Brandt et al. \cite{BrandtFaudreeGoddard}]\label{ThBrFaGo}
Every non-bipartite graph $G$ with $\delta(G)\geq(\nu(G)+2)/3$ is
weakly pancyclic and has girth 3 or 4.
\end{theorem}

\begin{theorem}[Jackson \cite{Jackson}]\label{ThJa}
Let $G$ be a bipartite graph with partition sets $X$ and $Y$,
$2\leq|X|\leq|Y|$. If for every vertex $x\in X$,
$d(x)\geq\max\{|X|,|Y|/2+1\}$, then $G$ has a cycle containing all
vertices in $X$, (i.e., of length $2|X|$).
\end{theorem}

A graph $G$ is said to be \emph{$k$-regular} if every vertex of $G$
has degree $k$.

\begin{lemma}\label{LeRegular}
Let $k$ and $n$ be two integers with $n\geq k+1$ and $k$ or
$n$ is even. Then there is a $k$-regular graph of order $n$ each
component of which is of order at most $2k+1$.
\end{lemma}

\begin{proof}
We first assume that $k+1\leq n\leq 2k+1$. If $k$ is even, then let
$G$ be the graph with vertex set $\{v_1,v_2,\ldots,v_n\}$ and every
vertex $v_i$ is adjacent to the $k$ vertices in $\{v_{i\pm
1},v_{i\pm 2},\ldots,v_{i\pm k/2}\}$, where the subscripts are taken
modulo $n$. Then $G$ is a $k$-regular graph of order $n$. If $k$ is
odd, then $n$ is even and $n-1-k$ is even. Similarly as above we can
get a $(n-1-k)$-regular graph $H$ of order $n$. Then
$G=\overline{H}$ is a $k$-regular graph of order $n$. Since $n\leq
2k+1$, every component of $G$ has order at most $2k+1$.

Now we assume that $n\geq 2k+2$.

If $k$ is even, then let
$$n=q(2k+1)+r,\ 0\leq r\leq 2k.$$
Note that $q\geq 1$. If $r=0$, then the union of $q$ copies of a
$k$-regular graph of order $2k+1$ is a required graph. If $k+1\leq
r\leq 2k$, then the union of $q$ copies of a $k$-regular graph of
order $2k+1$ and one copy of a $k$-regular graph of order $r$ is a
required graph. Now we assume that $1\leq r\leq k$. Note that
$k+1\leq k+r\leq 2k$. Then the union of $q-1$ copies of a
$k$-regular graph of order $2k+1$, one copy of a $k$-regular graph
of order $k+1$, and one copy of a $k$-regular graph of order $k+r$,
is a required graph.

If $k$ is odd, then $n$ is even. Let
$$n=2qk+r,\ 0\leq r<2k.$$
Clearly $r$ is even. If $r=0$ then the union of $q$ copies of a
$k$-regular graph of order $2k$ is a required graph. If $k+1\leq
r<2k$, then the union of $q$ copies of a $k$-regular graph of order
$2k$ and one copy of a $k$-regular graph of order $r$ is a required
graph. Now we assume that $2\leq r\leq k-1$. Note that $k+1\leq
k+r-1\leq 2k$. Then the union of $q-1$ copies of a $k$-regular graph
of order $2k$, one copy of a $k$-regular graph of order $k+1$, and
one copy of a $k$-regular graph of order $k+r-1$, is a required
graph.
\end{proof}

\section{Proof of Theorem \ref{ThLower}}

For convenience we define a constant $\theta$ such that $\theta=1$ if
both $n$ and $m/2$ are even, and $\theta=0$ otherwise. We will
construct a graph $G$ of order $2n+m/2-\theta-1$ such that $G$
contains no $K_{1,n}$ and $\overline{G}$ contains no $W_m$.

It is easy to check that $m/2-1$ or $n+m/2-\theta-1$ is even.
By Lemma \ref{LeRegular}, Let $H$ be an $(m/2-1)$-regular graph of
order $n+m/2-\theta-1$ such that each component of which has order
at most $m-1$. Let $G=\overline{H}\cup K_n$. Then
$\nu(G)=2n+m/2-\theta-1$.

We first show that $G$ contains no $K_{1,n}$. Clearly $K_n$ contains
no $K_{1,n}$. Note that every vertex in $H$ has degree $m/2-1$, and
then every vertex in $\overline{H}$ has degree
$\nu(H)-1-m/2+1=n-\theta-1$. Thus $\overline{H}$ contains no
$K_{1,n}$.

Second we show that $\overline{G}$ contains no $W_m$. Suppose to
contrary that $\overline{G}$ contains a $W_m$. Let $x$ be the hub of
the $W_m$. If $x$ is contained in $K_n$, then all vertices of the
wheel other than $x$ are in $V(H)$. This implies that $H$ has a cycle
$C_m$. But every component of $H$ has order less than $m$, a
contradiction. So we assume that $x\in V(H)$. Note that $x$ has
$m/2-1$ neighbors in $H$. At least $m/2+1$ vertices of the wheel are
in the $K_n$. This implies that there are two vertices in the $K_n$
such that they are adjacent in $\overline{G}$, a contradiction.

This implies that $R(K_{1,n},W_m)\geq 2n+m/2-\theta$.

\section{Proof of Theorem \ref{ThExactly}}

Note that by our assumption $n\geq 4$ and $m\geq 6$. We already
showed $R(K_{1,n},W_m)\geq 2n+m/2-\theta$ in Theorem \ref{ThLower}.
Now we prove that $R(K_{1,n},W_m)\leq 2n+m/2-\theta$. Let $G$ be a
graph of order
$$\nu(G)=2n+m/2-\theta.$$ Suppose that $\overline{G}$ has no
$K_{1,n}$, i.e.,
\begin{align}\label{EqdeltaG}
\delta(G)\geq n+m/2-\theta.
\end{align}
We will prove that $G$ has a $W_m$. We assume to the contrary that
$G$ contains no $W_m$. We choose such a $G$ with minimum size.

Let $u$ be a vertex of $G$ with maximum degree. Set
$$H=G[N(u)] \mbox{ and } I=V(G)\backslash(\{u\}\cup N(u)).$$
Note that $\nu(H)=d(u)$.

\begin{claim}\label{ClnuH}
$d(u)\geq n+m/2$; and for every $v\in V(H)$, $d(v)=n+m/2-\theta$.
\end{claim}

\begin{proof}
If $\theta=0$, then by (\ref{EqdeltaG}), $d(u)\geq n+m/2$. If
$\theta=1$, then $n$ and $m/2$ are both even. Thus $\nu(G)=2n+m/2-1$
is odd. If every vertex of $G$ has degree $2n+m/2-1$, then $G$ has
an even order, a contradiction. This implies $d(u)\geq n+m/2$.

Let $v$ be a vertex in $H$. Clearly $d(v)\geq\delta(G)\geq
n+m/2-\theta$. If $d(v)\geq n+m/2-\theta+1$, then $d(u)\geq d(v)\geq
n+m/2-\theta+1$. Thus $G'=G-uv$ has size less than $G$ with
$\delta(G')\geq n+m/2-\theta$. Since $G'$ is a subgraph of $G$, it
contains no $W_m$, a contradiction.
\end{proof}

By Claim \ref{ClnuH}, we assume that
\begin{align}\label{EqnuH}
\nu(H)=n+m/2+\tau, \mbox{ where } \tau\geq 0.
\end{align}

\begin{claim}\label{CldeltaH}
$\delta(H)\geq m/2+\tau$.
\end{claim}

\begin{proof}
Let $v$ be an arbitrary vertex of $H$. By Claim \ref{ClnuH},
$d(v)=n+m/2-\theta$. Note that
$\nu(G-H)=(2n+m/2-\theta)-(n+m/2+\tau)=n-\theta-\tau$. Thus
$$d_H(v)\geq d(v)-\nu(G-H)=(n+m/2-\theta)-(n-\theta-\tau)=m/2+\tau.$$
Thus the claim holds.
\end{proof}

\begin{claim}\label{ClSeparable}
$H$ is separable.
\end{claim}

\begin{proof}
By (\ref{EqnuH}), $\nu(H)\geq m\geq 3$. Suppose to contrary that $H$
is 2-connected. By Claim \ref{CldeltaH} and Theorem \ref{ThDi},
$c(G)\geq m$. Also note that $$3\delta(H)\geq 3m/2+3\tau\geq
n+m/2+3\tau+2\geq\nu(H)+2,$$ i.e., $\delta(H)\geq(\nu(H)+2)/3$.

If $H$ is non-bipartite, then by Theorem \ref{ThSuBa}, $H$ is weakly
pancyclic and of girth 3 or 4. Thus $H$ contains $C_m$. Note that
$u$ is adjacent to every vertex of the $C_m$, hence $G$ contains a $W_m$,
a contradiction.

If $H$ is bipartite, say with partition sets $X$ and $Y$, then $|X|\geq m/2+\tau$ and
$$|Y|=\nu(H)-|X|\leq(n+m/2+\tau)-(m/2+\tau)=n,$$
since $\delta(H)\geq m/2+\tau$.
Let $X'$ be a subset of $X$ with $|X'|=m/2$. Note that for every vertex $x$ of $X'$,
$$d_Y(x)=d_H(x)\geq m/2\geq n/2+1\geq|Y|/2+1.$$ By Theorem
\ref{ThJa}, the subgraph of $H$ induced by $X'\cup Y$ contains a
$C_m$. Thus $G$ contains a $W_m$, a contradiction.
\end{proof}

If $H$ is disconnected, then $H$ has at least two components; if $H$
is connected, then $H$ has at least two end-blocks. Now let $D$ be a
component or an end-block of $H$ such that $\nu(D)$ is as small as
possible. We define a constant $\varepsilon$ such that
$\varepsilon=1$ if $D$ is an end-block of $H$, and $\varepsilon=0$
otherwise. Thus
\begin{align}\label{EqnuD}
\nu(D)\leq(\nu(H)+\varepsilon)/2.
\end{align}

If $D$ is an end-block of $H$, then let $z$ be the cut-vertex of $H$
contained in $D$.

\begin{claim}\label{ClNeighbor}
For every two vertices $v,w\in V(D)$ which are not cut-vertices of
$H$, $|N_I(v)\cap N_I(w)|\geq m/2-1$.
\end{claim}

\begin{proof}
Note that $d_I(v)=d(v)-1-d_H(v)\geq d(v)-\nu(D)$, and $d_I(w)\geq
d(w)-\nu(D)$.
\begin{align*}
|N_I(v)\cap N_I(w)| & \geq d_I(v)+d_I(w)-|I|\geq d(v)+d(w)-2\nu(D)-|I|\\
    & \geq 2\delta(G)-(\nu(H)+\varepsilon)-|I|=2\delta(G)-\nu(G)+1-\varepsilon\\
    & =2(n+m/2-\theta)-(2n+m/2-\theta)+1-\varepsilon\\
    & =m/2+1-\theta-\varepsilon\geq m/2-1.
\end{align*}
Thus the claim holds.
\end{proof}

Suppose that there is a vertex $v\in V(D)$ which is not a cut-vertex
of $H$ such that $v$ has $m/2$ neighbors in $V(D)$ each of which is
not a cut-vertex of $H$. Then let $X$ be the set of such $m/2$
neighbors of $v$ and $Y=\{u\}\cup N_I(v)$. Let $B$ be the bipartite
subgraph of $G$ with partition sets $X$ and $Y$, and for any two
vertices $x\in X$ and $y\in Y$, $xy\in E(B)$ if and only if $xy\in
E(G)$.

Note that $|X|=m/2$. By Claim \ref{ClNeighbor}, every vertex of $X$
has at least $m/2$ neighbors in $Y$. By Claims \ref{ClnuH} and
\ref{CldeltaH}, $d(v)=n+m/2-\theta$ and $d_H(v)\geq m/2+\tau$. Thus
$|Y|=d(v)-d_H(v)\leq n-\theta-\tau$. Since $m\geq n+2$,
$m/2\geq|Y|/2+1$. By Theorem \ref{ThJa}, $B$ contains a $C_m$. Note
that $v$ is adjacent to every vertex of the $C_m$, hence $G$ has a $W_m$,
a contradiction.

So we conclude that $D$ is an end-block of $H$ (i.e.,
$\varepsilon=1$), and every vertex $v\in V(D)\backslash\{z\}$ has at
most $m/2-1$ neighbors in $V(D)\backslash\{z\}$. By Claim
\ref{CldeltaH}, we can see that $z$ is adjacent to every vertex in
$V(D)\backslash\{z\}$ and every vertex in $V(D)\backslash\{z\}$ has
degree in $H$ exactly $m/2$ and $\tau=0$.

\begin{claim}
Every vertex in $V(D)\backslash\{z\}$ is adjacent to every vertex in
$I$.
\end{claim}

\begin{proof}
Let $v$ be a vertex in $V(D)\backslash\{z\}$. Since
$d(v)=n+m/2-\theta$ and $d_H(v)=m/2$. we have
$$d_I(v)=d(v)-1-d_H(v)=n-1-\theta.$$
Also note that
$$|I|=\nu(G)-1-\nu(H)=(2n+m/2-\theta)-1-(n+m/2)=n-1-\theta.$$
This implies that $v$ is adjacent to every vertex in $I$.
\end{proof}

\begin{case}
  $N_I(z)\neq\emptyset$.
\end{case}

Note that $|I|=n-1-\theta\geq m/2-1$. Let $v \in V(D)\backslash\{z\}$ and $u_1,u_2,\ldots,u_{m/2-1}$
be $m/2-1$ vertices in $I$ such that $zu_1\in E(G)$, and let
$v_1,v_2\ldots,v_{m/2-1}$ be $m/2-1$ vertices in
$N_D(v)\backslash\{z\}$. Then $uzu_1v_1u_2v_2\cdots
u_{m/2-1}v_{m/2-1}u$ is a $C_m$. Since $v$ is adjacent to every
vertex of the $C_m$, $G$ contains a $C_m,$ a contradiction.

\begin{case}
  $N_I(z)=\emptyset$ and $G[I]$ is not empty.
\end{case}

Let $v \in V(D)\backslash\{z\}$ and $u_1,u_2,\ldots,u_{m/2-1}$ be
$m/2-1$ vertices in $I$ such that $u_1u_2\in E(G)$, and let
$v_1,v_2\ldots,v_{m/2-1}$ be $m/2-1$ vertices in
$N_D(v)\backslash\{z\}$. Then $uzv_1u_1u_2v_2u_3v_3\cdots$
$u_{m/2-1}v_{m/2-1}u$ is a $C_m$. Since $v$ is adjacent to every
vertex of the $C_m$, $G$ contains a $C_m,$ a contradiction.

\begin{case}
  $N_I(z)=\emptyset$ and $G[I]$ is empty.
\end{case}

Let $w$ be an arbitrary vertex in $I$. Note that $w$ is nonadjacent
to every vertex in $\{u,z\}\cup I$. Hence
$$d(w)\leq\nu(G)-2-|I|=(2n+m/2-\theta)-2-(n-1-\theta)=n+m/2-1.$$
Since $d(w)\geq\delta(G)=n+m/2-\theta$, we can see that $\theta=1$
and $w$ is adjacent to every vertex of $V(H)\backslash\{z\}$.
Moreover, every vertex in $I$ is adjacent to every vertex in
$V(H)\backslash\{z\}$.

Since $\theta=1$, by Claim \ref{ClnuH}, $d(u)=n+m/2$ and
$d(z)=n+m/2-1$. Thus there is a vertex $x\in V(H)\backslash\{z\}$
such that $xz\notin E(G)$. By Claim \ref{CldeltaH}, let
$v_1,v_2,\ldots,v_{m/2}$ be $m/2$ vertices in $N_H(x)$ and
$u_1,u_2,\ldots,u_{m/2}$ be $m/2$ vertices in $\{u\}\cup I$. Then
$u_1v_1u_2v_2\cdots u_{m/2}v_{m/2}u_1$ is a $C_m$. Since $x$ is
adjacent to every vertex of the $C_m$, $G$ contains a $W_m$, a
contradiction.

The proof is complete.

\end{document}